\documentclass{article}
\usepackage{amssymb}
\usepackage{amsmath}
\usepackage{amsthm}
\usepackage{enumerate}
\usepackage{mathabx}
\usepackage{hyperref}

\theoremstyle{plain}

\newtheorem{Thm}{Theorem}
\newtheorem{Cor}{Corollary}

\newcommand*{\R}{\ensuremath{\mathbb{R}}}
\newcommand*{\Z}{\ensuremath{\mathbb{Z}}}
\newcommand*{\C}{\ensuremath{\mathbb{C}}}

\newcommand*{\Ann}{\ensuremath{\mathrm{Ann\,}}}

\newcommand*{\Hom}{\ensuremath{\mathrm{Hom\,}}}

\begin{document}
	
	\date{}
	
\author{{L\'aszl\'o Sz\'ekelyhidi}\\
{\small\it Institute of Mathematics, University of Debrecen,}\\
{\small\rm e-mail: \tt lszekelyhidi@gmail.com}}   
	
	\title{Spectral Synthesis on Varieties
		\footnotetext{The research was supported by the
			Hungarian National Foundation for Scientific Research (OTKA),
			Grant No. K134191}
	\footnotetext{AMS (2000)
			Subject Classification:
			43A45, 43A25, 13N15}}

\maketitle

\section{Abstract}
In his classical paper \cite{MR0023948}, L.~Schwartz proved that on the real line, in every linear translation invariant space of continuous complex valued functions, which is closed under compact convergence the exponential monomials span a dense subspace. He studied so-called local ideals in the space of Fourier transforms, and his proof based on the observation that, on the one hand, these local ideals are completely determined by the exponential monomials in the space, and, on the other hand, these local ideals completely determine the space itself. On the other hand, D.~I.~Gurevich in \cite{MR0390759} gave counterexamples for Schwartz's theorem in higher dimension. In this paper we show that the ideas of localization can be extended to general locally compact Abelian groups using abstract derivations on the Fourier algebra of compactly supported measures. Based on this method we present necessary and sufficient conditions for spectral synthesis for varieties on locally compact Abelian groups. Using localization, in \cite{MR4789359} we proved that spectral synthesis holds on a locally compact Abelian group $G$ if and only if it holds on $G/B$, where $B$ is the closed subgroup of compact elements. This leads to a complete characterization of locally compact Abelian groups having spectral synthesis  in \cite{Sze23d}.

\section{Introduction}
If $G$ is a locally compact Abelian group, then its {\it measure algebra} is the space $\mathcal M_c(G)$ of all compactly supported complex Borel measures on $G$. It is well-known that $\mathcal M_c(G)$ can be identified with the topological dual of the space $\mathcal C(G)$, the space of all continuous complex valued functions on $G$ equipped with the pointwise operations and with the topology of uniform convergence on compact sets. The space $\mathcal M_c(G)$ is equipped with the {\it convolution} of measures, defined in the usual way:
$$
\langle \mu*\nu,f\rangle=\int \int f(x+y)\,d\mu(x)\,d\nu(y)
$$
whenever $\mu,\nu$ are in $\mathcal M_c(G)$ and $f$ is in $\mathcal C(G)$. Convolution makes $\mathcal M_c(G)$ a commutative topological agebra with unit, if endowed with the weak*-topology. On the other hand, $\mathcal C(G)$ is a topological vector module over $\mathcal M_c(G)$ with the action
$$
\mu*f(x)=\int f(x-y)\,d\mu(y)
$$
whenever $\mu$ is in $\mathcal M_c(G)$, $f$ is in $\mathcal C(G)$ and $x$ is in $G$. 

The continuous homomorphisms of $G$ into the multiplicative topological group of nonzero complex numbers 
are called {\it exponentials}. 
For each measure $\mu$ in $\mathcal M_c(G)$, the {\it Fourier--Laplace transform} (shortly: {\it Fourier transform}) of the measure $\mu$ is defined on the set of exponentials as
$$
\widehat{\mu}(m)=\int \widecheck{m}\,d\mu.
$$
Here $\widecheck{m}(x)=m(-x)$. The mapping $\mu\mapsto \widehat{\mu}$ is one-to-one, and its image $\mathcal A(G)$ is called the {\it Fourier algebra} of $G$. It is easy to check that $\mathcal A(G)$ is an algebra: the mapping $\mu\mapsto \widehat{\mu}$ is a linear isomorphism, and by the identity
$$
(\mu*\nu)\,\widehat{}\,=\widehat{\mu}\cdot \widehat{\nu},
$$
it is an algebra isomorphism.  As $\mathcal A(G)$ is isomorphic to the measure algebra, it is reasonable to equip it with the topology induced by all Fourier transforms. For the sake of convenience, we shall denote a general element of $\mathcal A(G)$ by $\widehat{\mu}$, where $\mu$ is a measure in $\mathcal M_c(G)$, and similarly, a general subset of $\mathcal A(G)$ will be referred to as $\widehat{H}$, where $H$ is a subset of $\mathcal M_c(G)$.
The ideals of the Fourier algebra are in close connection with the translation invariant subspaces of $\mathcal C(G)$. A closed translation invariant linear subspace of $\mathcal C(G)$ is called a {\it variety}. The smallest nonzero variety is one dimensional: it is spanned by a single  exponential. For each $f$ in $\mathcal C(G)$, the intersection of all varieties including $f$ is called {\it the variety of $f$}, and is denoted by $\tau(f)$. The  annihilator of each variety is a closed ideal in $\mathcal M_c(G)$, and conversely, every closed ideal in $\mathcal M_c(G)$ is the annihilator of some variety in $\mathcal C(G)$. We have the basic identity
$$
\Ann \Ann V=V\enskip\text{and}\enskip \Ann \Ann I=I
$$
for each variety $V$ in $\mathcal C(G)$ and for every closed ideal $I$ in $\mathcal M_c(G)$ (see e.g. \cite{MR3502634}). This variety-ideal connection induces a similar one-to-one connection between varieties on the group and closed ideals in the Fourier algebra. We note that, in general, not every ideal in the measure algebra, resp. in the Fourier algebra is closed (see e.g. \cite{MR2340978}). On the other hand, if $G$ is a discrete group, then every ideal in the measure algebra, resp. in the Fourier algebra, is closed (see \cite{MR3185617}). 

Maximal ideals of the Fourier algebra will play an important role in our discussion. Obviously, the annihilators of one dimensional varieties are closed maximal ideals. They are called {\it exponential maximal ideals}. The annihilator of the one dimensional variety spanned by the exponential $m$ is \hbox{denoted by $M_m$}. Clearly, the corresponding maximal ideal $\widehat{M}_m$ is the set of all Fourier transforms which vanish at $m$. 

Spectral analysis and synthesis studies the structure of varieties. We say that {\it spectral analysis} holds for the variety $V$, if every nonzero subvariety of $V$ contains a nonzero finite dimensional variety. This is equivalent to the property that every nonzero subvariety of $V$ contains an exponential. We say that the variety $V$ is {\it synthesizable}, if all finite dimensional subvarieties in $V$ span a dense subspace. Finally, we say that {\it spectral synthesis} holds for a variety, if all of its subvarieties are synthesizable. If spectral analysis, resp. spectral synthesis holds for each variety, then we say that {\it spectral analysis}, resp. {\it spectral synthesis holds on the group}. 

Synthesizability of a variety means that every function in the variety is the uniform limit on compact sets of functions which belong to some finite dimensional subvariety. Functions in finite dimensional varieties are called {\it exponential polynomials}. The complete description of exponential polynomials can be found in \cite{MR3185617}. We note that continuous homomorphisms of $G$ into the additive group of complex numbers are called {\it additive functions}.

\begin{Thm}\label{expol}
	Let $G$ be a locally compact Abelian group. The continuous function $f:G\to \C$ is an exponential polynomial if and only if it has the form
	$$
	f(x)=\sum_{j=1}^n P_j\big(a_1(x),a_2(x),\dots,a_k(x)\big)m_j(x)
	$$
	where $a_i$ is a real additive function, $m_j:G\to \C$ is an exponential, and \hbox{$P_j:\C^k\to \C$} is a  polynomial $(i=1,2,\dots,k; j=1,2,\dots,n)$.
\end{Thm}

The exponential polynomial $f$ above is called an {\it $m_1$-exponential monomial}, or briefly an {\it exponential monomial} if $n=1$, and it is called a {\it polynomial}, if $n=1$ and $m_1= 1$. In other words, a polynomial on $G$ is a polynomial of additive functions, that is, it has the form $x\mapsto P\big(a_1(x),a_2(x),\dots,a_k(x)\big)$.  
Exponential monomials can be characterized by difference operators. For each exponential $m$ and for every $y$ in $G$ we let $\Delta_{m;y}=\delta_{-y}-m(y)\delta_0$, $\delta_g$ being the point mass supported at the point $g$ in $G$. It can be shown that these measures generate the exponential maximal ideal $M_m$ (see \cite{MR3502634}). We use the notation
$$
\Delta_{m;y_1,y_2,\dots,y_n}=\Delta_{m;y_1}*\Delta_{m;y_2}*\cdots *\Delta_{m;y_n}.
$$
In particular,
$$
\Delta_{1;y_1,y_2,\dots,y_n}=\Delta_{y_1,y_2,\dots,y_n}
$$
whenever $y_1,y_2,\dots,y_n$ are in $G$. It is easy to see that
\begin{equation}\label{poly0}
\Delta_{m;y_1,y_2,\dots,y_n}*(fm)=m(y_1+\cdots+y_n)(\Delta_{y_1,y_2,\dots,y_n}*f)\cdot m.
	\end{equation}

\begin{Thm}\label{expmonchar}
	Let $G$ be a locally compact Abelian group and $m$ an exponential. The continuous function $f:G\to\C$ is an $m$-exponential monomial if and only if it is contained in a finite dimensional variety, and there is a natural number $n$ such that
	\begin{equation}\label{expmoneq}
		\Delta_{m;y_1,y_2,\dots,y_{n+1}}*f=0
	\end{equation}
	holds for each $y_1,y_2,\dots,y_{n+1}$ in $G$.
\end{Thm}

If we drop the condition that $f$ is contained in a finite dimensional variety, then the solutions of \eqref{expmoneq} are called {\it generalized exponential $m$-monomials}. In particular, if $m=1$, then the solutions of \eqref{expmoneq} are called {\it generalized polynomials}.
The smallest natural number $n$  which  satisfies \eqref{expmoneq} is called the {\it degree} of $f$. 

\section{Differential operators on the Fourier algebra}
Let $G$ be a locally compact abelian group. The set of all exponentials on $G$ is denoted by $\widetilde{G}$: clearly, it is an abelian group with respect to pointwise multiplication of functions. It is obvious that  $\widetilde{G}$ is isomorphic to the direct product $\widehat{G}\times \Hom(G,\R)$, where $\widehat{G}$ is the Pontrjagin dual of $G$, and $\Hom(G,\R)$ denotes the set of all real additive functions on $G$. Indeed, if $m$ is an exponential on $G$, we define 
$$
\chi(x)=\frac{m(x)}{|m(x)|}
\hskip.5cm\text{and}\hskip.5cm
a(x)=\ln |m(x)|
$$ 
for each $x$ in $G$. Clearly, $\chi$ is a character  of $G$, and $a$ is a real additive function, further
$$
m(x)=\chi(x)\cdot \exp a(x)
$$
is a unique representation of $m$.
It follows that the mapping $m\mapsto (\chi,a)$ is a bijective homomorphism between $\widetilde{G}$ and $\widehat{G}\times \Hom(G,\R)$. It is reasonable to equip $\widetilde{G}$ with the compact-open topology. We note that $\widetilde{G}$  is a topological group, but in some cases it is not locally compact. Nevertheless, if $G$ is compactly generated, then $\Hom(G,\R)$ is a finite dimensional vector space. The topological group $\widetilde{G}$ is called the {\it generalized dual} of $G$.
\vskip.2cm

Given a topological group a continuous homomorphisms of the additive group of $\C$ into this group is called a {\it one-parameter group}. We are interested in one-parameter groups in the generalized dual group $\widetilde{G}$. We recall the following result from \cite{MR3481109}:

\begin{Thm}\label{onepar}
	Let $G$ be a locally compact Abelian group. For every additive function $a:G\to \C$ the function $\Gamma_a:\C\to \widetilde{G}$ defined by
	\begin{equation}
		\Gamma_a(z)(x)=\exp z a(x)
	\end{equation}
	for $x$ in $G$ and $z$ in $\C$, is a one-parameter group in $\widetilde G$. Conversely, for every one-parameter group $\Gamma:\C\to \widetilde{G}$ there exists a unique additive function $a:G\to\C$ such that $\Gamma=\Gamma_a$.
\end{Thm}

It is known (see e.g. \cite{MR3481109} and the references there) that a kind of differential calculus can be introduces on generalized dual groups. Here we recall the relevant concepts. Given a function $f:\widetilde{G}\to \C$, an exponential $m$ in $\widetilde{G}$ and a one-parameter group $\Gamma$ in $:\widetilde{G}$ we define  {\it the derivative of $f$ at $m$ along $\Gamma$} as the limit
$$
\partial_{\Gamma}f(m)=\lim_{z\to 0}\, \lim_{w\to 0} \frac{1}{w} [f\bigl(m\cdot \Gamma(z+w)\bigr)-f\bigl(m\cdot \Gamma(z)\bigr)],
$$
if this limit exists. In the light of the previous theorem $\Gamma=\Gamma_a$ for some $a$, hence we can use the simpler notation $\partial_af(m)$, where $a$ is an additive function. If this limit exists at each $m$ in $\widetilde{G}$, then the function $\partial_af: m\mapsto \partial_af(m)$ is defined on $\widetilde{G}$, which can be called the {\it derivative of $f$ along $a$}, or the {\it partial derivative of $f$ with respect to $a$}. Repeating this process for $\partial_af$ instead of $f$ we define higher order partial derivatives assuming that the corresponding limits exist. If $f$ has the property that its partial derivatives exist for any order with respect to any additive function, then $f$ is called a {\it $\mathcal C^{\infty}$-function}. By Theorem 12 in \cite{MR3481109}, every Fourier transform is a $\mathcal C^{\infty}$-function.
\vskip.2cm

We denote higher order derivatives by the symbol  $\partial_{a_1}^{\alpha_1} \partial_{a_2}^{\alpha_2}\dots \partial_{a_n}^{\alpha_n}f$ symbol, where $a_1,a_2,\dots,a_n$ are additive functions, and $\alpha_1,\alpha_2,\dots,\alpha_n$ are nonnegative integers. In addition, if $P$ is a complex polynomial in $n$ variables, then the meaning of $P(\partial_{a_1},\partial_{a_2},\dots,\partial_{a_n})f$ is also clear. We call the operators of the form $P(\partial)=P(\partial_{a_1},\partial_{a_2},\dots,\partial_{a_n})$ {\it differential operators on the Fourier algebra} (see \cite{MR3481109}). We note that the differential operator $P(\partial_{a_1},\partial_{a_2},\dots,\partial_{a_n})$ acts on the function $\widehat{\mu}$ by the following formula:
\begin{equation}\label{poldif}
P(\partial_{a_1},\partial_{a_2},\dots,\partial_{a_n})\widehat{\mu}(m)=\int P(\widecheck{a}_1,\widecheck{a}_2,\dots,\widecheck{a}_n) \widecheck{m}\,d\mu.
\end{equation}
Clearly, all differential operators on $\widetilde{G}$ form a commutative unital algebra, denoted by $\mathcal P(G)$, which is ismorphic to the algebra of all polynomials \hbox{on $G$.}
\vskip.2cm

This formula suggests that we can introduce a more general class of differential operators using generalized polynomials. The formal definition, according to \eqref{poldif}, is as follows:
\begin{equation}\label{genpoldif}
	D\widehat{\mu}(m)=\int \widecheck{p}_D \widecheck{m}\,d\mu,
\end{equation} 
where $p_D$ is a generalized polynomial on $G$, which we call the {\it generating function} of the operator $D$. We shall call such operators {\it generalized differential operators}. Clearly, every differential operator we defined above is a generalized differential operator -- we shall call them {\it polynomial differential operators}. All generalized differential operators form a commutative unital algebra, denoted by $\mathcal D(G)$, which is isomorphic to the algebra of generalized polynomials on $G$.
In order to justify the name "differential operator" we note that all generalized differential operators are derivations on the Fourier algebra in the following sense: the identity operator is considered a derivation of order $0$. The continuous linear operator $D$ on the Fourier algebra is called a {\it derivation of order $1$}, if it satisfies
$$
D(\widehat{\mu}\cdot  \widehat{\nu})=D(\widehat{\mu})\cdot  \widehat{\nu}+\widehat{\mu}\cdot D( \widehat{\nu})
$$
for each $\widehat{\mu}, \widehat{\nu}$ in $\mathcal A(G)$. For every positive integer $n$ the continuous linear operator $D$ on the Fourier algebra is called a {\it derivation of order $n+1$}, if the mapping 
$$
(\widehat{\mu}, \widehat{\nu})\mapsto D(\widehat{\mu}\cdot  \widehat{\nu})-D(\widehat{\mu})\cdot  \widehat{\nu}-\widehat{\mu}\cdot D( \widehat{\nu})
$$
is a derivation of order $n$ in both variables. It is easy to check that the operator $D$ defined in \eqref{genpoldif} is a derivation of order $n$ if and only if its generating function $p_D$ is a generalized polynomial of degree $n$, that is, a solution of the functional equation \eqref{expmoneq} with $m=1$. We shall see that the existence of non-polynomial generalized derivations is closely related to the failure of spectral synthesis.

\section{Localizable ideals}
Given an ideal $\widehat{I}$ in $\mathcal A(G)$ the exponential $m$ is called a {\it root of $\widehat{I}$}, if $\widehat{\mu}(m)=0$ holds for each $\widehat{\mu}$ in $\widehat{I}$. The set of all roots of $\widehat{I}$ is denoted by $Z(\widehat{I})$.
\vskip.2cm

The idea of localization is based on the observation that exponential monomials of the form $x\mapsto P(x)m(x)$ in a variety $V=\Ann I$ can be considered as differential operators $P(-D)$ annihilating $\widehat{I}$ at $m$. Spectral synthesis for the variety means that there are sufficiently many annihilating differential operators at each root such that they characterize the ideal. This idea will be worked out in the subsequent paragraphs.
\vskip.2cm

We say that the generalized differential operator $D_0$ with generating function $p_{D_0}$ {\it annihilates the set $H$ in $\mathcal A$ at $m$}, if for each generalized differential operator $D$ whose generating function $p_D$ belongs to the variety $\tau(p_{D_0})$ of $p_{D_0}$ satisfies
$$
D\widehat{\mu}(m)=0
$$
whenever $\widehat{\mu}$ is in $H$. For instance, the differential operator $D=id$ annihilates an ideal exactly at the roots of the ideal. Given a set $\mathcal D$ of generalized differential operators the set of all $\widehat{\mu}$'s which are annihilated by each $D$ in $\mathcal D$ at the exponential $m$ is denoted by $\widehat{I}_{\mathcal D,m}$. It is easy to check that $\widehat{I}_{\mathcal D,m}$ is a closed ideal in the Fourier algebra. For instance, given an exponential $m$ and $\mathcal D=\{id\}$, then $\widehat{I}_{\mathcal D,m}$ is the maximal ideal $\widehat{M}_m$ consisting of all Fourier transforms which vanish at the exponential $m$. 
\vskip.2cm

A dual concept is defined as follows. Given an ideal $\widehat{I}$ in $\mathcal A(G)$, the set of all generalized differential operators $D$ on $\mathcal A(G)$ annihilating $\widehat{I}$ at the exponential $m$ is denoted by $\mathcal D_{\widehat{I},m}$, and the set of all polynomial differential operators $P(\partial)$  annihilating $\widehat{I}$ at $m$ is denoted by $\mathcal P_{\widehat{I},m}$ We note that $\mathcal D_{\widehat{I},m}=\{0\}$, if and only if $m$ is not a root of the ideal $\widehat{I}$. It is easy to see that $\mathcal D_{\widehat{I},m}$ is a linear space which is translation invariant. In particular, $\mathcal P_{\widehat{I},m}$ is invariant under differentiation, which means that if $P(\partial)$ is in $\mathcal P_{\widehat{I},m}$, then $\partial^{\alpha}P(\partial)$ is in $\mathcal P_{\widehat{I},m}$, as well.
Clearly, for each exponential $m$, we have the obvious relation
$\widehat{I}\subseteq  \widehat{I}_{\mathcal D_{{\widehat{I},m}},m}\subseteq \widehat{I}_{\mathcal P_{{\widehat{I},m}},m},$
which implies 
\begin{equation}\label{neq}
	\widehat{I}\subseteq \bigcap_{m\in Z(\widehat{I})}\widehat{I}_{\mathcal D_{{\widehat{I},m}},m}\subseteq \bigcap_{m\in Z(\widehat{I})}\widehat{I}_{\mathcal P_{{\widehat{I},m}},m}.
\end{equation}
It is clear, that if $m$ is not a root of $\widehat{I}$, then $\mathcal D_{{\widehat{I},m}}=\mathcal P_{{\widehat{I},m}}=\{0\}$, hence  we have $\widehat{I}_{\mathcal D_{{\widehat{I},m}},m}=\widehat{I}_{\mathcal P_{{\widehat{I},m}},m}=\mathcal A(G)$, consequently such $m$'s have no effect on the intersection.
\vskip.2cm
We say that the ideal $\widehat{I}$ is {\it localizable}, if it has following property:  if  $\widehat{\nu}$ satisfies $P(\partial)(\widehat{\nu})(m)=0$ for each exponential $m$ and for every $P(\partial)$ in $\mathcal P_{\widehat{I},m}$, then $\widehat{\nu}$ \hbox{is in $\widehat{I}$.}  In other words, localizability of $\widehat{I}$ means that we have equality in equation \eqref{neq}. Roughly speaking, localizable ideals are uniquely determined by their roots together with their "multiplicity". For instance, exponential maximal ideals are localizable. Indeed, if $\widehat{M}_m$ is non-localizable, then there exists a $\widehat{\nu}$ not in $\widehat{M}_m$ such that  
$$
P(\partial)(\widehat{\nu})(m)=0
$$
for each polynomial differential operator $P(\partial)$ in $\mathcal P_{\widehat{M}_m,m}$. As the identity operator $id$ is in  $\mathcal P_{\widehat{M}_m,m}$, we immediately have \hbox{$\widehat{\nu}(m)=0$,} consequently $\widehat{\nu}$ is in $\widehat{M}_m$, a contradiction.
\vskip.2cm

Another simple but important observation is that if there exists a non-polynomial differential operator on the Fourier algebra which annihilates $\widehat{I}$ at some exponential, that is, if there exists a generalized polynomial in $\Ann I$, which is not a polynomial, then, by \eqref{neq}, the ideal $\widehat{I}$ may not be localizable, because the two intersections in \eqref{neq} may be different.

\section{Spectral synthesis and localizability}

\begin{Thm}
	Let $\mathcal D$ be a family of generalized differential operators on $\mathcal A(G)$.  The closed ideal $\widehat{I}$ in $\mathcal A(G)$ has the property
	\begin{equation}\label{idinc}
		\widehat{I}= \bigcap_{m\in Z(\widehat{I})} \widehat{I}_{\mathcal D,m}
	\end{equation}
if and only if the functions $\widecheck{f}_{D,m}m$ with $D$ in $\mathcal D$ and $m$ in $Z(\widehat{I})$, span a dense subspace in $\Ann I$.
\end{Thm}

\begin{proof}
Let $\widehat{J}=\bigcap_{m\in Z(\widehat{I})} \widehat{I}_{\mathcal D,m}$, and assume that $\widehat{J}\subseteq \widehat{I}$. If the subspace spanned by all functions of the form $\widecheck{f}_{D,m}m$ with $D$ in $\mathcal D$ and $m$ in $Z(\widehat{I})$ is not dense in $\Ann I$, then there exists a $\mu_0$ not in $\Ann \Ann I=I$ such that $\mu_0$  annihilates all functions of the form $\widecheck{f}_{D,m}m$ with $D$ in $\mathcal D$ and $m$  in $Z(\widehat{I})$. In other words, for each $x$ in $G$ we have
$$
0=(\mu_0*\widecheck{f}_{D,m})m(x)=\int \widecheck{f}_{D,m}(x-y)m(x-y)\,d\mu_0(y)=
$$
$$
\int f_{D,m}(y-x)\widecheck{m}(y-x)\,d\mu_0(y)
=m(x)\int f_{D,m}(y-x)\widecheck{m}(y)\,d\mu_0(y).
$$
In particular, for $x=0$
$$
0=\mu_0*\widecheck{f}_{D,m}m(0)=\int f_{D,m}(y)\widecheck{m}(y)\,d\mu_0(y)=D(\mu_0)(m)
$$
holds for each $D$ in $\mathcal D$ and for every root $m$ of $\widehat{I}$.  In other words, $\widehat{\mu}_0$ is in $\widehat{I}_{\mathcal{D},m}$ for each root $m$ of $\widehat{I}$, hence it is in the set $\widehat{J}$, but not in $\widehat{I}$ -- a contradiction.
\vskip.2cm

Conversely, assume that the subspace spanned by all functions of the form $\widecheck{f}_{D,m}m$ with $D$ in $\mathcal D$ and $m$ in $Z(\widehat{I})$, is dense in $\Ann I$.  It follows that any $\mu$ in $\mathcal M_c(G)$, which satisfies
\begin{equation}\label{dercond0}
	\int \widecheck{f}_{D,m}(x-y)m(x-y)\,d\mu(y)=0
\end{equation}
for all $D$ in $\mathcal D$, $m$ in $Z(\widehat{I})$ and $x$ in $G$, belongs to $I=\Ann \Ann I$. Now let $\widehat{\mu}$ be in $\widehat{I}_{\mathcal D,m}$, and suppose that $D$ is in $\mathcal D$. Then for each $x$ in $G$ and for every root $m$ of $\widehat{I}$, the function $\widehat{\mu}\cdot \widehat{\delta}_{-x}$ is in $\widehat{I}_{\mathcal D,m}$, hence 
$$
0=D(\widehat{\mu}\cdot \widehat{\delta}_{-x})(m)=\int \widecheck{f}_{D,m}(x-y)m(x-y)\,d\mu(y),
$$
that is, $\widehat{\mu}$ satisfies \eqref{dercond0} for each $D$ in $\mathcal D$ and $m$ in $Z(\widehat{I})$. This implies that $\mu$ is in $I$, and the theorem is proved.
\end{proof}

\begin{Cor}\label{derid5}
	Let $\widehat{I}$ be an ideal in $\mathcal A(G)$. Then $\widehat{I}=\bigcap_{m} \widehat{I}_{\mathcal P_{\widehat{I},m},m}$ holds if and only if all  functions of the form $\widecheck{f}_{D,m}m$ with $m$ in $Z(\widehat{I})$ and $D$ in $\mathcal P_{\widehat{I},m}$ span a dense subspace in the variety $\Ann I$.
\end{Cor}

This corollary clearly implies our main result: 

\begin{Cor}\label{syn}
	Let  $\widehat{I}$ be an ideal in $\mathcal A(G)$. Then $\Ann I$ is  synthesizable if and only if $\widehat{I}$ is localizable.
\end{Cor}

\begin{Cor}\label{nonloc}
	If $D$ is a non-polynomial derivation and $m$ is an exponential, then $\widehat{I}_{D,m}$ is non-localizable.
\end{Cor}

\begin{proof}
	Suppose that $D$ is of order $n$, and let $f_{D,m}$ be the generating function of $D$. As $D$ is non-polynomial, hence $n\geq 2$. Then $\widecheck{f}_{D,m}m$ is a generalized exponential monomial of degree $n$, and it cannot be annihilated by any polynomial derivation of degree at most $n-1$. It follows that $\mathcal P_{\widehat{I}_{D,m},m}$ consists of polynomial derivations of order at most $n-1$. Consequently, all  functions of the form $\widecheck{f}_{D,m}m$ with $m$ in $Z(\widehat{I})$ and $D$ in $\mathcal P_{\widehat{I},m}$ do not span a dense subspace in the variety $\Ann I_{D,m}$. We infer that $\Ann I_{D,m}$ is not synthesizable, and, by the previous theorem, $\widehat{I}_{D,m}$ is non-localizable.
\end{proof}

We note that, by the structure theory of locally compact Abelian groups, every derivation on the Fourier algebra of compactly generated  locally compact Abelian groups is polynomial. On the other hand, in the non-compactly generated case there always exist non-polynomial derivations on the Fourier algebra, even in the discrete case (see \cite{MR2039084,MR2340978}). 
\vskip.2cm

An application of the concept of localizability, we can give a simple proof for the following theorem (see \cite{MR2340978, MR3703555}):

\begin{Cor}\label{tors}
	Spectral synthesis holds on a discrete Abelian group if and only if its torsion free rank is finite.
\end{Cor}

\begin{proof}
	If the torsion free rank of $G$ is infinite, then there is a generalized polynomial on $G$, which is not a polynomial (see \cite{MR2039084}), hence there is a non-polynomial derivation on the Fourier algebra. By Corollary \ref{nonloc}, there exists a non-localizable ideal $\widehat{I}$ in the Fourier algebra of $G$, hence the variety $\Ann I$ is not synthesizable.
	\vskip.2cm
	
	Conversely, let $G$ have finite torsion free rank. The subgroup $B$ of compact elements of $G$ coincides with the set $T$ of all elements of finite order, and $G/T$ is a (continuous) homomorphic image of $\Z^n$ with some nonnegative integer $n$. As spectral synthesis holds on $\Z^n$ (see \cite{MR0098951}), it holds on its homomorphic images, by the results in \cite{MR3589184}. Finally, by the results in \cite{MR4789359}, if spectral synthesis holds on $G/T$, then it holds on $G$, as well.
\end{proof}

\section{Statements and Declarations}
Data sharing not applicable to this article as no datasets were generated or analysed during the current study. There are no financial or non-financial interests that are directly or indirectly related to the work submitted for publication.

\end{document}